\documentclass[10pt,draft,a4paper]{amsart}
\usepackage{amssymb}
\usepackage{graphicx}
\numberwithin{equation}{section}
\newtheorem{theorem}{Theorem}[section]

\newtheorem{lemma}[theorem]{Lemma}
\theoremstyle{remark}
\newtheorem{remark}{Remark}[section]

\theoremstyle{definition}

\newcommand{\R}{\mathbb{R}}
\newcommand{\C}{\mathbb{C}}

\begin{document}

\title%
[Counterexamples to Strichartz estimates]%
{On the lack of dispersion for a class of magnetic Dirac flows}
\begin{abstract}
  We show that global Strichartz estimates for magnetic Dirac operators generally fails, if the potentials do not decay fast enough at infinity. In order to prove this, we construct some explicit examples of homogeneous magnetic potentials with less than Coulomb decay, i. e. with homogeneity-degree more than -1, such that the magnetic field points to a fixed direction, which does not depend on $x\in\R^3$.
\end{abstract}
\date{\today}    
\author{Naiara Arrizabalaga}
\address{Naiara Arrizabalaga: Universidad del Pa\'is Vasco, Departamento de
Matem\'aticas, Apartado 644, 48080, Bilbao, Spain}
\email{naiara.arrizabalaga@ehu.es}

\author{Luca Fanelli}
\address{Luca Fanelli:
Universidad del Pa\'is Vasco, Departamento de
Matem\'aticas, Apartado 644, 48080, Bilbao, Spain}
\email{luca.fanelli@ehu.es}

\author{Andoni Garc\'ia}
\address{Andoni Garc\'ia:
Universidad del Pa\'is Vasco, Departamento de
Matem\'aticas, Apartado 644, 48080, Bilbao, Spain}
\email{andoni.garcia@ehu.es}


\maketitle


\section{Introduction}\label{sec:introd}

Among the differential models in Quantum Mechanics, a relevant and interesting role is played by the Dirac system. The free Dirac operator (in the standard Pauli representation) is the $1^{\text{st}}$-order differential operator
\begin{equation*}
H_0:=-i\alpha\cdot\nabla+m\beta=
-i\sum_{k=1}^3\alpha_k\partial_k+m\beta,
\end{equation*}
where $\alpha_k,\beta\in\mathcal M_{4\times4}(\C)$, $k=1,2,3$, are the {\it Dirac matrices}
\begin{equation}\label{eq:matrici}
\alpha_k=
  \left(
  \begin{array}{cc}
    0 & \sigma_k
    \\
    \sigma_k & 0
  \end{array}\right),
  \qquad
   \beta
  =
  \left(
  \begin{array}{cc}
    I_2 & 0
    \\
    0 & -I_2
  \end{array}\right),
  \qquad
  k=1,2,3
\end{equation}
defined in terms of the {\it Pauli matrices} $\sigma_k\in\mathcal M_{2\times2}(\C)$, given by
\begin{equation*}
I_2
=
  \left(
  \begin{array}{cc}
    1 & 0
    \\
    0 & 1
  \end{array}\right),
  \quad
  \sigma_1
  =
  \left(
  \begin{array}{cc}
    0 & 1
    \\
    1 & 0
  \end{array}\right),
  \quad
  \sigma_2
  =
  \left(
  \begin{array}{cc}
    0 & -i
    \\
    i & 0
  \end{array}\right),
  \quad
  \sigma_3
  =
  \left(
  \begin{array}{cc}
    1 & 0
    \\
    0 & -1
  \end{array}\right);
\end{equation*}
the operator $H_0$ acts on spinor-valued functions $f:\R^3\to\C^4$. Since the matrices $\alpha_k,\beta$ are hermitian, $H_0$ is self-adjoint on the Hilbert space $L^2(\R^3;\C^4)$, with domain $H^1(\R^3;\C^4)$ (see e.g. \cite{T}).
The Dirac system reads as follows: 
\begin{equation}\label{eq:d}
  \partial_tu =iH_0u.
\end{equation}
Here $u=u(t,x):\R^{1+3}\to\C^4$, and we neglected the physical constants. Notice that the system \eqref{eq:d} is (weakly) hyperbolic, since $\alpha_k,\beta$ have eigenvalues $\pm1$, both with multiplicity 2; consequently, equation \eqref{eq:d} has finite speed of propagation, according to the causality principle of General Relativity.

The motivation which led Dirac to introduce \eqref{eq:d} in the well known paper \cite{D} was to describe the evolution of the free electron in the 3D-space, taking into account the spin of the particle. Since he needed to factorize the Laplace operator in a suitable differential square root, he was forced to look for (hermitian) matrix-coefficients satisfying the well known anti-commutation rules
\begin{equation*}
\alpha_l\alpha_k+\alpha_k\alpha_l=2\delta_{kl}I_4;
\end{equation*}
the above defined set $\{\alpha_k,\beta\}$ in \eqref{eq:matrici} is one of the possible examples of  linearly independent matrix-sets satisfying the previous properties. One of the consequences of the anti-commutation rules is that the square of $H_0$ is a diagonal matrix of Laplace operators, namely
\begin{equation*}
  H_0^2=(-\Delta+m^2)I_4,
\end{equation*}
where $I_4$ is the $4\times4$-identity matrix. Therefore, by conjugating
$i\partial_t-H_0$, one immediately obtains the Klein-Gordon operator, i.e.
\begin{equation}\label{eq:KG}
  (i\partial_t+H_0)(i\partial_t-H_0)=\left(-\partial_t^2+\Delta-m^2\right)I_4.
\end{equation}
This shows that \eqref{eq:d} can be listed within the class of {\it dispersive equations}.

In the last few years, thanks to the research on nonlinear models (as nonlinear Schr\"odinger and wave, Korteweg-de Vries etc...), it has been understood that the dispersion (when it is present) plays a fundamental role in the dynamics. A great effort has been indeed devoted to study the tools which permit to quantify dispersive phaenomena in terms of a priori estimates for the free or perturbed flows. 

The interest in looking to the Dirac equation as a dispersive model is in fact a recent matter of research. Notice that, due to 
\eqref{eq:KG}, if one considers sufficiently regular initial data, the unitary flow $e^{itH_0}$ (which uniquely defines the solution to \eqref{eq:d}, when it acts on functions in the domain of $H_0$) satisfies the same dispersive estimates as the ones for the 3D Klein-Gordon equation. Among them, a particular attention has been devoted to Strichartz estimates (see the standard references \cite{GV}, \cite{KT}, \cite{S}), which in this case are
\begin{equation}\label{eq:strimass}
\|e^{itH_0}f\|_{L^p_tL^q_x}\leq C\|f\|_{H^{\frac{1}{p}-\frac{1}{q}+\frac{1}{2}}},
\end{equation}  
for any couple $(p,q)$ satisfying the Schr\"odinger admissibility condition 
\begin{equation}\label{eq:admismass}
\frac{2}{p}+\frac{3}{q}=\frac{3}{2},\qquad 2\leq p\leq\infty,\qquad 2\leq q\leq6.
\end{equation}
Here we used the standard notations
\begin{equation*}
  \|f\|_{H^{s}}:=\|\langle D\rangle^s f\|_{L^2},
  \qquad
  \|f\|_{\dot H^s}:=\| |D|^sf\|_{L^2},
  \qquad
  \langle D\rangle:=\left(1+|D|^2\right)^{\frac12},
\end{equation*}
where $|D|=\mathcal F^{-1}(|\xi|\mathcal F)$ and $\mathcal F$ is the usual Fourier transform.
Moreover we denoted by
\begin{equation*}
  \|f\|_{L^p_tL^q_x}:=\left(\int_\R\left(\int_{\R^3}|f|^q\,dx\right)^{\frac pq}\,dt\right)^{\frac1p}.
\end{equation*}
For the {\it massless} Dirac equation (i.e. $m=0$, the {\it neutrino}-model), the analogous estimates are
\begin{equation}\label{eq:stri}
\|e^{it\mathcal{D}}f\|_{L^p_tL^q_x}\leq C\|f\|_{\dot{H}^{\frac{1}{p}-\frac{1}{q}+\frac{1}{2}}},
\end{equation}
for any couple $(p,q)$ satisfying the wave-admissibility condition
\begin{equation}\label{eq:admis}
\frac{2}{p}+\frac{2}{q}=1,\qquad 2<p\leq\infty, \qquad 2\leq q<\infty.
\end{equation}
Here we denoted by
\begin{equation*}
  \mathcal D:= -i\alpha\cdot\nabla = -i\sum_{k=1}^3\alpha_k\partial_k,
\end{equation*}
the massless Dirac operator.
Inequalities \eqref{eq:strimass} and \eqref{eq:stri} follow by the Strichartz estimates for the Klein-Gordon and wave equation, respectively (see e.g. \cite{DF} for details). We remark that the endpoint estimate $p=2$ is not included in \eqref{eq:stri}, in analogy with the 3D-wave equation for which it fails, as proved in \cite{KM} (see also \cite{M} for the Schr\"odinger case). As \eqref{eq:strimass} and \eqref{eq:stri} show, a natural loss of derivatives with respect to the initial data is needed (except to the case $(p,q)=(\infty,2)$, which obviously follows by the unitarity of the groups $e^{itH_0}, e^{it\mathcal D}$); indeed, due to the finite speed of propagation, no smoothing effect (in the sense of the $L^q_x$-regularity) can occur. 

A relevant question is whether the dispersion, and in particular Strichartz estimates, is preserved or not under rough linear perturbations of the free operator. The interaction of a free particle with an external field is usually modeled, in the case of Dirac operators, by perturbating the principal part with a 0-order term.
More precisely, the perturbed (massive) Dirac operator takes the following form
\begin{equation*}
  H=H_0+V,
\end{equation*}
where we assume $V=V(t,x):\R^{1+3}\to\mathcal M_{4\times 4}(\C)$ to be a hermitian matrix $V^t=\overline{V}$, in order to preserve the symmetry.
If $V=V(x)$ and $H$ is self-adjoint, then the unitary group $e^{itH}$ can be standardly defined via Spectral Theorem. In the last few years, some efforts have been spent in order to understand the dispersive properties of perturbed Dirac propagators. In \cite{DF2}, the authors prove some time-decay estimates for the massless Dirac equation, if the potential decay sufficiently fast at infinity, and is possibly singular at some point. Nevertheless, these estimates are far from being optimal, so that they cannot be used in order to prove Strichartz estimates using the standard techniques by Ginibre-Velo \cite{GV}, and Keel-Tao \cite{KT}. On the other hand, some perturbative techniques have been introduced in order to obtain Strichartz estimates from some {\it weakly dispersive estimates}, involving local energy decay and Morawetz estimates, which can be proved by direct methods. Inspired by the ideas in \cite{RV} first, and \cite{ST} \cite{BPST} \cite{BPST2} later, in which a suitable mix of free Strichartz and local smoothing is used in a $TT^*$-argument for the Schr\"odinger equation with an electric potential or more in general with variable coefficients, in \cite{DF} the authors prove the full range of Strichartz estimates for both the propagators $e^{itH}, e^{it(\mathcal D+V)}$, in the case of short-range potentials.

One of the main difficulties in handling Dirac operators with potential is that the square is not generally diagonal, differently from the free case $V\equiv0$. Moreover, a $1^{\text{st}}$-order term naturally appears in the expansions of $H^2$ and $\mathcal D^2$. Therefore, it is natural to look for potentials $V$ with some specific structure, in order to get some precise informations on the dynamics by direct techniques.
Among the possible models of matrix-vector fields $V$, the so called {\it magnetic potentials} possess some relevant features.
A magnetostatic potential is a vector field
\begin{equation*}
  A=A(x)=(A^1(x),A^2(x),A^3(x)):\R^{3}\to\R^{3}.
\end{equation*}
The magnetic Dirac operator take the following form:
\begin{equation}\label{eq:dirmag}
  H=-i\alpha\cdot(\nabla-iA)+m\beta=-i\sum_{k=1}^3\alpha_k\left(\partial_k-iA^k\right)+m\beta,
\end{equation}
where the connexion changes, passing from straight derivatives to the covariant ones. We will often use the standard notation
\begin{equation*}
  \nabla_A:=\nabla-iA.
\end{equation*}
In the massless case, we will usually denote by
\begin{equation}\label{eq:dirmag2}
\mathcal D_A:=-i\alpha\cdot\nabla_A=-i\sum_{k=1}^3\alpha_k(\partial_k-iA^k).
\end{equation}
We also introduce the {\it magnetic field}, which is given by
\begin{equation*}
  B(x)=\text{curl}A(x).
\end{equation*}
One interesting property of these operators are the following identities
\begin{equation*}
   H^2=(-\Delta_A+m)I_4-2S\cdot B,
   \qquad
   \mathcal D_A^2=-\Delta_AI_4-2S\cdot B;
\end{equation*}
here we denoted by $\Delta_A=(\nabla-iA)^2$, while $S$ is the spin-operator, given by
\begin{equation*}
  S=\frac i4\alpha\wedge\alpha=\frac i4(\alpha_2\alpha_3-\alpha_3\alpha_2,
  \alpha_3\alpha_1-\alpha_1\alpha_3, \alpha_1\alpha_2-\alpha_2\alpha_1).
\end{equation*}
All the previous identities are formal (see \cite{BDF} for details), and can be justified on 
$\mathcal C^{\infty}_0(\R^3;\C^4)$, then extended by density to the domains of the operators $H^2,\mathcal D_A^2$. This shows that, apart from the 0-order term involving the spin, the principal part of $H^2$ and $\mathcal D_A^2$ is a diagonal matrix of magnetic Laplace operators. As a consequence, the structure of the magnetic Dirac equations
\begin{equation*}
  \partial_tu=iHu,
  \qquad
  \partial_tu=i\mathcal D_Au
\end{equation*}
is strictly related to the magnetic Klein-Gordon and wave equations.
This fact was exploited in \cite{BDF}, in which the authors can prove weak dispersive estimates by direct techniques, involving multiplier methods; then they can use it to show that Strichartz estimates hold, in the full admissibility range.
This follows a program which is common to several results, produced in the last few years by many different authors, about the dispersive properties of magnetic Schr\"odinger and wave equations (see \cite{DFVV}, \cite{EGS1}, \cite{EGS}, \cite{FV}, \cite{GST}, \cite{G}, \cite{Ste}).

The authors of the above mentioned papers focused their attention to potentials which are rough from the point of view of the Sobolev regularity, but need to decay sufficiently at infinity and are possibly singular at points. In all these cases, the Coulomb-type potentials (namely $|A(x)|\sim1/|x|$) appear as a natural threshold for the validity of global (in time) Strichartz estimates. By the way, the only heuristic argument to corroborate this claim is given by the scaling invariance of the massless Dirac equation. Notice that the equation $\partial_tu=i\mathcal Du$ is invariant under the scaling
\begin{equation*}
  u_\lambda(t,x)=u\left(\frac t\lambda,\frac x\lambda\right);
  \qquad
  \lambda>0.
\end{equation*}
if we impose this property to be true for the magnetic Dirac equation $\partial_tu=i\mathcal D_Au$, we are forced to consider potentials which are homogeneous of degree $-1$, as in the Coulomb case. In view of these considerations, it is a natural problem to search for potentials with less than Coulomb decay at infinity, for which Strichartz estimates fail.
A first result in this direction has been proved by Goldberg, Vega and Visciglia in \cite{GVV}. In that case, the authors consider the electric Schr\"odinger equation $\partial_tu=i(\Delta+V(x))u$, with a repulsive potential $V$ with a non-degenerate critical point, of the form
\begin{equation*}
  V(x)=|x|^{-\delta}\omega\left(\frac{x}{|x|}\right),
\end{equation*}
in any dimension $n\geq1$: they prove that in the range $0<\delta<2$ the complete set of Strichartz estimates fail. The main idea is to show that, if $0<\delta<2$, most of the mass of the solution is localized around a non-dispersive function, namely a standing wave generated by an eigenfunction of a suitable harmonic oscillator. The approximation is performed by Taylor expanding the coefficients of the equation around the degenerate direction of the potential 
$V$. 
Later on, Duyckaerts \cite{D} showed some counterexamples to global Strichartz estimates for the same equation, by constructing some compactly supported potentials with bad singularities at points.
For the magnetic Schr\"odinger equation $\partial_tu=i\Delta_Au$, in \cite{FG} the authors constructed some counterexamples to Strichartz estimates, in any dimension $n\geq3$, based on potentials of the form
\begin{equation}\label{eq:potenziali}
  A(x)=|x|^{-\delta}M_nx;
\end{equation}
here $M_n\in\mathcal M_{n\times n}(\R)$ is the anti-symmetric matrix 
\begin{equation}\label{eq:matrix}
  M_{2k+1}=
  \left(
  \begin{array}{ccccc}
    \Omega_2 & 0 & \cdots & 0 & 0
    \\
    0 & \Omega_2 & \cdots & 0 & 0
    \\
    \vdots & \vdots & \ddots & \vdots & \vdots 
    \\
    0 & 0 & \cdots & \Omega_2 & 0
    \\
    0 & 0 & \cdots & 0 & 0 
  \end{array}
  \right),
  \quad
  M_{2k}=
  \left(
  \begin{array}{cccccc}
    \Omega_2 & 0 & \cdots & 0 & 0 & 0
    \\
    0 & \Omega_2 & \cdots & 0 & 0 & 0
    \\
    \vdots & \vdots & \ddots & \vdots & \vdots & \vdots 
    \\
    0 & 0 & \cdots & \Omega_2 & 0 & 0
    \\
    0 & 0 & \cdots & 0 & 0 & 0
    \\
    0 & 0 & \cdots & 0 & 0 & 0
  \end{array}
  \right),
\end{equation}
in odd and even dimension respectively, with $k\geq2$, and 
\begin{equation*}
\Omega_2=
\left(
\begin{array}{cc}
0 & 1 
\\ 
-1 & 0
\end{array}
\right).
\end{equation*}
In this case, it is crucial that the potential $A$ has at least one degenerate direction, as suggested by \eqref{eq:matrix}. As a consequence, the 2D-magnetic case cannot be treated in \cite{FG}, and at our knowledge it remains an open question whether it is possible or not to construct explicit counterexamples to Strichartz estimates. 
The idea is in fact inspired to the one by Goldberg, Vega and Visciglia in \cite{GVV}, but the examples in \cite{FG} are somehow more natural. Indeed, the expansion of $\Delta_A$ gives
\begin{equation*}
  \Delta_A=(\nabla-iA)^2=\Delta-2iA\cdot\nabla-|A|^2,
\end{equation*}
since $\text{div}A\equiv0$. Here the role of $V$ is played by $|A|^2$, and the non-degenerate critical point is given by $P=(0,0,1)$, which is in fact suggested by the fix direction of the magnetic field $B$.

The aim of this paper is to prove that the analogous examples also contradict the Strichartz estimates \eqref{eq:stri} for the (massless) Dirac equation. Our main result is the following.
\begin{theorem}\label{thm:counterdirac}
Let us consider the following anti-symmetric $3\times3$ matrix 
\begin{equation}\label{eq:M}
 M
  :=
  \left(
  \begin{array}{ccc}
    0 & 1 & 0
    \\
     -1 & 0 & 0
     \\ 0 & 0 & 0
  \end{array}\right),
\end{equation}
and A be the following vector field:
 \begin{equation}\label{eq:A}
    A(x)=|x|^{-\delta}Mx,
    \qquad
    1<\delta<2.
  \end{equation}
  Then the solution of the magnetic Dirac equation 
  \begin{equation}\label{eq:Diracmagnintr}
  \begin{cases}
    i\partial_tu(t,x)+\mathcal{D}_Au(t,x)=0
    \\
    u(0,x)=f(x),
  \end{cases}
\end{equation}
   with initial datum 
  $f\in \dot H^{\frac1p-\frac1q+\frac12}$, does not satisfy the Strichartz estimates \eqref{eq:stri}, for any admissible couple $(p,q)\neq(\infty,2)$ as in \eqref{eq:admis}.
\end{theorem}
\begin{remark}\label{rem:1}
  For potentials as in \eqref{eq:A}, the operator $\mathcal D_A$ is self-adjoint on $L^2$ (see the standard reference \cite{T}). Consequently, the unique solution of 
  \eqref{eq:Diracmagnintr} is defined as $u(t,\cdot)=e^{it\mathcal D_A}f(\cdot)$, by the Spectral Theorem.
\end{remark}
\begin{remark}\label{rem:2}
  Notice that, in the range $\delta\leq1$, the potentials given by \eqref{eq:A} do not decay at infinity. As a consequence, the spectrum of $\mathcal D_A$ is purely discrete, so that any standing wave $u(t,x)=e^{it\lambda}Q(x)$, where $\lambda$ is an eigenvalue of $\mathcal D_A$ and 
  $Q$ a corresponding eigenfunction is a solution of \eqref{eq:Diracmagnintr} which cannot verify any global Strichartz estimate. Moreover, in the range $\delta>2$ Strichartz estimates hold (see Theorem 1.6 in \cite{BDF}. The question about the critical behavior $\delta=2$ still remains open.
\end{remark}
\begin{remark}\label{rem:3}
  It is possible to prove that the above potentials are also counterexamples to the Strichartz estimates 
  \eqref{eq:strimass} for the massive Dirac equation. Indeed, one should repeat the same proof of Theorem \ref{thm:counterdirac} and take into account the lower order terms which appear once computing the non-homogeneous Sobolev norms.
\end{remark}
\begin{remark}[Magnetic waves]\label{rem:4}
  Strictly related to the Dirac equation, one could consider the Cauchy problem for the magnetic wave equation
  \begin{equation}\label{eq:wave}
  \begin{cases}
    \partial_t^2u-\Delta_Au=0
    \\
    u(0)=u_0
    \\
    \partial_tu(0)=u_1.
  \end{cases}
  \end{equation}
  Here the solution $u=u(t,x):\R^{1+n}\to\C$ can be expressed as 
  \begin{equation*}
    u(t,\cdot)=
    \cos\left(t\sqrt{-\Delta_A}\right)u_0(\cdot)+
    \frac{\sin\left(t\sqrt{-\Delta_A}\right)}{\sqrt{-\Delta_A}}u_1(\cdot),
  \end{equation*}
  where $-\Delta_A$ is assumed to be self-adjoint and positive and $\sqrt{-\Delta_A}$ is defined via Spectral Theorem. Therefore, in order to estimate the solution of \eqref{eq:wave} it suffices to prove that the magnetic wave propagator $e^{it\sqrt{-\Delta_A}}$ is bounded between suitable Banach spaces. In the free case $A\equiv0$, the Strichartz estimates are
  \begin{equation}\label{eq:striwave}
\|e^{it\sqrt{-\Delta}}f\|_{L^p_tL^q_x}\leq C\|f\|_{\dot{H}^{\frac{1}{p}-\frac{1}{q}+\frac{1}{2}}},
\end{equation}
for any couple $(p,q)$ satisfying the wave-admissibility condition
\begin{equation}\label{eq:admiswave}
  \frac{n-1}{p}+\frac{n-1}{q}=\frac{n-1}2,
  \qquad 
  2\leq p\leq\infty, 
  \qquad 
  (p,q)\neq(2,\infty),
\end{equation}
in analogy with \eqref{eq:stri}. We can prove the following result.
\begin{theorem}\label{thm:counterwave}
Let $n\geq3$ and $A$ be the vector field
 \begin{equation}\label{eq:AA}
    A(x)=|x|^{-\delta}Mx,
    \qquad
    1<\delta<2,
  \end{equation}
  where $M$ is given by \eqref{eq:matrix}.
  Then the solution of \eqref{eq:wave}, with initial datum 
  $f\in \dot H^{\frac1p-\frac1q+\frac12}$, does not satisfy the Strichartz estimates 
  \eqref{eq:striwave}, for any admissible couple $(p,q)\neq(\infty,2)$ as in \eqref{eq:admiswave}.
\end{theorem}
The proof of Theorem \ref{thm:counterwave} is completely analogous to the one of Theorem \ref{thm:counterdirac}. Indeed, the only difference is that the examples in \eqref{eq:AA} depend on the dimension, so that the scalings which we perform in Section \ref{sec:reduction} need some natural modification. We omit here further details; see also \cite{FG} in which the analogous case for the magnetic Schr\"odinger equation is handled.
\end{remark}
The rest of the paper is devoted to the proof of Theorem \ref{thm:counterdirac}. In section 
\ref{sec:reduction}, we study the operator 
\begin{equation*}
T:=-i\alpha\cdot\nabla_y-\alpha\cdot(y_2,-y_1,0)
\end{equation*}
and we produce some suitable standing wave for the Dirac evolution with constant magnetic field, describing its properties in Lemma \ref{lem:lemapre}. In section \ref{sec:proof} we perform the proof of the main theorem, showing that the solution of \eqref{eq:Diracmagnintr} is mostly concentrated around the above mentioned standing wave, if the potential behaves like in \eqref{eq:A}.

\section{An approximating operator}\label{sec:reduction}
The main idea in the proof of Theorem \ref{thm:counterdirac} is to show that most of the mass of the solution of \eqref{eq:Diracmagnintr} is concentrated around a non-dispersive function, namely a standing wave for a suitable dimensionless operator, which is suggested by the explicit form of the magnetic field. In this section, we introduce the fundamental tools which allow us to reduce, by suitable space-time localizations, to such a situation. 

In the following, we always denote by 
\begin{equation*}
  x=(y,z)\in\R^3,
  \qquad
  y=(y_1,y_2)\in\R^2,
  \qquad
  z\in\R.
\end{equation*}
Let consider the operator 
\begin{equation}
T:=-i\alpha\cdot\nabla_y-\alpha\cdot(y_2,-y_1,0).
\end{equation}
Notice that $T$ is a (massless) Dirac operator with constant magnetic field $B=(0,0,2)$.
As it is well known (see \cite{T}, Section 7.1.3 for details),
$T$ has compact resolvent and in particular its spectrum reduces to a discrete set of eigenvalues. Moreover, the eigenfunctions $v=v(y)$, solving
\begin{equation}\label{eq:eigenproblem}
Tv(y)=\lambda v(y),
\end{equation}
for some eigenvalue $\lambda\in\R$, have exponential decay, and in addition 
\begin{equation}\label{eq:lp}
  v(y) \in \bigcap _{p=1}^{\infty} L^{p}(\R^{2}),
  \qquad
  \Delta v \in \bigcap _{p=1}^{\infty} L^{p}(\R^{2}).
\end{equation}
Let us fix an eigenvalue $\lambda\in\R$, with a corresponding eigenfunction $v$, and define 
$\omega:\R^2\times(0,\infty)$ as
\begin{equation}\label{eq:omega1}
\omega(x) := v\left(\frac{y}{\sqrt{z^{\delta}}}\right),
\qquad
x:=(y,z) \in \mathbb{R}^{2} \times (0, \infty),
\end{equation}
where $\delta$ is the same as in \eqref{eq:A}.
By a direct computation, we see that $\omega$ satisfies 
\begin{equation}
 \left(-i\alpha\cdot\nabla_y-{z^{-\delta}}\alpha\cdot M(y,0)^t\right)\omega=\frac{\lambda}{\sqrt{z^{\delta}}} \omega,
\end{equation}
where $M$ is given by \eqref{eq:M}. Starting by $\omega$, we now create a standing wave $W(t,y,z)$ as follows: 
\begin{equation}\label{eq:epo5}
W(t,y,z) = e^{i( \lambda t /\sqrt{z^{\delta}})}\omega(y,z),
\qquad
(t,y,z) \in \mathbb{R} \times \mathbb{R}^{2} \times (0, \infty).
\end{equation}
By direct computations, it turns out that $W$ solves
\begin{equation}\label{eq:solves}
i\partial_{t}W - \left(i\alpha\cdot\nabla +{z^{-\delta}}\alpha\cdot M(y,0)^t\right) W
 = -i\alpha_3\partial_z W.
\end{equation}
Moreover,
\begin{align*}
\partial_z W&=\partial_z(e^{i(\lambda t/\sqrt{z^\delta})}\omega(y,z))
\\
&=-\frac{\delta}{2}\frac{i\lambda t}{z^{1+\delta/2}}e^{i(\lambda t/\sqrt{z^\delta})}v\left(\frac{y}{\sqrt{z^\delta}}\right)-\frac{\delta}{2}\frac{e^{i(\lambda t/\sqrt{z^\delta})}}{z^{1+\delta/2}}y\cdot\nabla v\left(\frac{y}{\sqrt{z^\delta}}\right),
\end{align*}
where the gradient in the last term of the previous identity is made with respect to the 2D variable $\frac{y}{\sqrt{z^{\delta}}}$.
Therefore we obtain by \eqref{eq:solves} that
\begin{equation}
i\partial_{t}W - \left(i\alpha\cdot\nabla +{z^{-\delta}}\alpha\cdot M(y,0)^t\right) W
 = F,
\end{equation}
where 
\begin{equation}\label{eq:F}
F(t,y,z)=-\frac{i e^{i(\lambda t/\sqrt{z^\delta})}}{z}\alpha_3\left\{-\frac{\delta}{2}\frac{i\lambda t}{z^{\delta/2}}v\left(\frac{y}{\sqrt{z^\delta}}\right)-\frac{\delta}{2}G\left(\frac{y}{\sqrt{z^\delta}}\right)\right\},
\end{equation}
with
\begin{equation}\label{eq:G}
G(y) = y \cdot \nabla_{y}v(y).
\end{equation}
We now introduce two real-valued cutoff functions $\psi, \chi \in C_0^\infty(\R)$ with the following properties:
\begin{align}\label{eq:psi}
&\psi(z) = 0\quad \text{for }|z| > 1,
\qquad \qquad  \qquad
\psi(z) = 1\quad \text{for }|z| < 3/4,\\
&\chi(z)=0\quad \text{for }|z|>1, |z|<1/4, 
\qquad 
\chi(z) = 1\quad \text{for }1/2<|z|<3/4.
\nonumber
\end{align}
Let us fix a parameter $\gamma \in (1/2,1)$, and for any $R>0$ denote by
\begin{equation}\label{eq:psiR}
  \psi_{R}(z) := \psi\left(\frac{z - R}{R^{\gamma}}\right).
\end{equation}
Finally, truncate $W$ as follows:
\begin{equation}\label{eq:WR}
 W_{R}(t,y,z) := W(t,y,z)\psi_R(z)\psi\left(\frac{|y|^{2}}{z^{2}}\right)\chi\left(\frac{|y|^{2}}{z^{2}}\right).
\end{equation}
Again, a direct computation shows that $W_R$ solves the Cauchy problem
\begin{equation}\label{eq:sistemaWR}
  \begin{cases}
    i\partial_tW_R-(i\alpha\cdot\nabla+z^{-\delta}\alpha\cdot M(y,0)^t)W_{R}
    
    =
    F_R
    \\
    W_{R}(0,y,z) = f_{R}(y,z),
  \end{cases}
\end{equation}
where the initial datum is given by
\begin{equation}\label{eq:fR}
f_{R}(y,z) = \psi_{R}(z)\psi\left(\frac{|y|^{2}}{z^{2}}\right)\chi\left(\frac{|y|^{2}}{z^{2}}\right)\omega(y,z).
\end{equation}
Moreover,
\begin{equation}\label{eq:FR}
F_{R}(t,y,z) =  \psi_{R}\psi\chi F + G_{R},
\end{equation}
where $F$ is given by \eqref{eq:F} and 
$G_R$ has the following form:
\begin{align}\label{eq:GR}
G_{R}(t,y,z) = &e^{i( \lambda t /\sqrt{z^{\delta}})}\left\{-\frac{2i}{z^2}\psi_R(\psi'\chi+\psi\chi')\omega(\alpha_1,\alpha_2)\cdot(y_1,y_2)
\right.
\\
&
\left.-i\alpha_3\psi'_R\psi\chi\omega+i\alpha_3\frac{2|y|^2}{z^3}\psi_R\psi'\chi\omega+i\alpha_3\frac{2|y|^2}{z^3}\psi_R\psi\chi'\omega\right\}.
\nonumber
\end{align}

The next lemma is the main result of this section.

\begin{lemma}\label{lem:lemapre}
Let $p,q\in(1,\infty)$, $\sigma=\frac{1}{p}-\frac{1}{q}+\frac{1}{2}$ and $\gamma\in(1/2,1)$; then 
\begin{equation}\label{eq:epo21}
 \|f_{R}\|_{\dot{H}^{\sigma}_{x}} \leq C R^{(\delta+\gamma)/2-\sigma\gamma},
\end{equation}
\begin{equation}\label{eq:epo22}
 \|W_{R}\|_{L^{p}_{T}L^{q}_{x}} \geq C T^{1/p} R^{(\delta+\gamma)/q},
\end{equation}
\begin{equation}\label{eq:epo23}
 \|F_{R}\|_{L^{p}_{T}\dot{H}^{2\sigma}_q} \leq C T^{1/p} R^{(\delta+\gamma)/q-2\sigma\gamma}\max\{R^{-\gamma}, TR^{-(1+\delta/2)},R^{-(2-\delta/2)}\},
\end{equation}
for all $R > 2, T > 0$, and some constant
$C = C(q,\gamma) > 0$. In particular, if $(p,q)\neq(\infty,2)$ is a an admissible couple in the sense of \eqref{eq:admis},
and $\beta > 0$, then the following estimates hold
\begin{equation}\label{eq:epo25}
\frac{\|W_{R}\|_{L^{p}((0,R^{\beta}); L^{q}_{x})}}{\|f_{R}\|_{\dot{H}^{\sigma}_{x}}} \geq CR^{(\beta - (\delta-\gamma))/p},
\end{equation}
\begin{equation}\label{eq:epo26}
\frac{\|W_{R}\|_{L^{p}((0,R^{\beta}); L^{q}_{x})}}{\|F_{R}\|_{L^{p'}((0,R^{\beta}); \dot{H}^{2\sigma}_{q'})}} \geq CR^{\kappa},
\end{equation}
for any $R>2$,
where
\begin{equation*}
\kappa = \kappa(\gamma, \beta, p) = 2\left(\frac{\beta  - (\delta-\gamma)}{p}\right) +\min\{\gamma - \beta, 1+\delta/2-2\beta, 2-\delta/2-\beta\}
\end{equation*}
and the constant $C > 0$ does not depend on $R$.
\end{lemma}
\begin{proof}
For a given function $L(y)$, denote by
\begin{equation}\label{eq:epo28}
\Lambda(y,z) := L\left(\frac{y}{\sqrt{z^{\delta}}}\right),
\end{equation}
with $(y,z) \in \mathbb{R}^{2} \times \mathbb{R}$.

{\bf Proof of \eqref{eq:epo22}.} It is sufficient to show that, if $0\neq L\in \bigcap _{p=1}^{\infty} L^{p}(\mathbb{R}^{2})$, then the following estimates hold:
\begin{equation}\label{eq:epo29}
cR^{(\delta+\gamma)/q} \leq  \|\Lambda\psi_{R}\psi\chi\|_{L^{q}_{x}} \leq CR^{(\delta+\gamma)/q},
\end{equation}
for all $R>1$,
where $c = c(q,L) > 0$ and $C = C(q,L) > 0$. Here $\psi$, $\chi$ and $\psi_{R}$ are defined by \eqref{eq:psi}, \eqref{eq:psiR}. Indeed, \eqref{eq:epo22} follows from \eqref{eq:epo29}, with the choice $L(y)=v(y)$. Notice that by the properties of $\psi_{R}$ and $\psi$ we have
\begin{align*}
& \int_{R-3R^{\gamma}/4}^{R+3R^{\gamma}/4}dz\int_{\frac{\sqrt{z^{2-\delta}}}{\sqrt{2}}<|y| < \frac{\sqrt{3}\sqrt{z^{2-\delta}}}{2}}|L(y)|^{q}z^{\delta}dy
\\
&\ \ \
= \int_{R-3R^{\gamma}/4}^{R+3R^{\gamma}/4}dz\int_{\frac{z}{\sqrt{2}}<|y| < \frac{\sqrt{3}z}{2}}|\Lambda|^{q}dy
\leq \int_{\mathbb{R}^{n}}|\Lambda\psi_{R}\psi\chi|^{q}dydz
\\
&\ \ \
\leq \int_{R-R^{\gamma}}^{R+R^{\gamma}}dz\int_{|y| < z}|\Lambda|^{q}dy
= \int_{R-R^{\gamma}}^{R+R^{\gamma}}dz\int_{|y| < \sqrt{z^{2-\delta}}}|L(y)|^{q}z^{\delta}dy,
\end{align*}
which implies \eqref{eq:epo29}, and consequently \eqref{eq:epo22}.

{\bf Proof of \eqref{eq:epo21}.}
First notice that, arguing as above we obtain
\begin{align*}
&cR^{(\delta+\gamma)/q} \leq  \|\Lambda\psi_{R}\psi^{'}\chi\|_{L^{q}_{x}} \leq CR^{(\delta+\gamma)/q}
\\
&cR^{(\delta+\gamma)/q - \gamma} \leq  \|\Lambda\psi^{'}_{R}\psi\chi\|_{L^{q}_{x}} \leq CR^{(\delta+\gamma)/q - \gamma}
\\
&cR^{(\delta+\gamma)/q} \leq  \|\Lambda\psi_{R}\psi\chi'\|_{L^{q}_{x}} \leq CR^{(\delta+\gamma)/q}.
\end{align*}
Now, \eqref{eq:epo29} implies that
\begin{equation}\label{eq:epo30}
 \|\Lambda\psi_{R}\psi\chi\|_{L^{2}_{x}}\leq C R^{(\delta+\gamma)/2}.
\end{equation}
We need now to estimate $\|\Lambda\psi_{R}\psi\chi\|_{\dot{H}^1_{x}}$; in order to do this, write
\begin{equation}\label{eq:norma12}
\nabla(\Lambda\psi_{R}\psi\chi)=\Lambda(\nabla\psi_R)\psi\chi+\Lambda\psi_R\nabla(\psi\chi)
+(\nabla\Lambda)\psi_R\psi\chi.
\end{equation}
For the first term in \eqref{eq:norma12}, we get
\begin{equation}\label{eq:gradpsiR}
\|\Lambda(\nabla\psi_R)\psi\chi\|_{L^2_x}=\|\Lambda\psi'_R\psi\chi\|_{L^2_x}\leq CR^{(\delta+\gamma)/2-\gamma}.
\end{equation}
The second term in \eqref{eq:norma12} can be treated analogously as follows:
\begin{equation}\label{eq:second}
\|\Lambda\psi_R\nabla(\psi\chi)\|_{L^2_x}\leq\|\Lambda\psi_R\nabla\psi\chi\|_{L^2_x}+\|\Lambda\psi_R\psi\nabla\chi\|_{L^2_x}.
\end{equation}
Denote by
\begin{equation}\label{eq:PSI}
   \Psi(y,z) := \frac{|y|}{\sqrt{z^{\delta}}} L\left(\frac{y}{\sqrt{z^{\delta}}}\right);
\end{equation}
since 
\begin{equation}
  |\nabla\psi|\leq\frac{|y|}{z^2}\psi',\quad |\nabla\chi|\leq\frac{|y|}{z^2}\chi',
\end{equation}
provided that $|y| < |z|$, we can estimate
\begin{align}
\|\Lambda\psi_R\nabla\psi\chi\|_{L^2_x}&\leq\|\frac{|y|}{z^2}\Lambda\psi_R\psi'\chi\|_{L^2_x}
\\
&\leq CR^{-(2-\delta/2)}\|\Psi\psi_R\psi'\chi\|_{L^2_x}\leq CR^{(\delta+\gamma)/2-(2-\delta/2)},
\nonumber
\end{align}
\begin{align}
\|\Lambda\psi_R\psi\nabla\chi\|_{L^2_x}&\leq\|\frac{|y|}{z^2}\Lambda\psi_R\psi\chi'\|_{L^2_x}
\\
&\leq CR^{-(2-\delta/2)}\|\Psi\psi_R\psi\chi'\|_{L^2_x}\leq CR^{(\delta+\gamma)/2-(2-\delta/2)}.
\nonumber
\end{align}
Therefore, by \eqref{eq:second} we obtain
\begin{equation}\label{eq:gradpsichi}
\|\Lambda\psi_R\nabla(\psi\chi)\|_{L^2_x}\leq CR^{(\delta+\gamma)/2-(2-\delta/2)}.
\end{equation}
We now study the last term in \eqref{eq:norma12}. By \eqref{eq:epo28},
\begin{align}\label{eq:soboleveigen}
\|(\nabla\Lambda)\psi_R\psi\chi\|_{L^2_x}&=\|z^{-\delta/2}(\nabla L)\left(\frac{y}{z^{\delta/2}}\right)\psi_R\psi\chi\|_{L^2_x}\\
&\leq CR^{-(a(1-\delta/2)+\delta/2)}\|\widetilde{M}\psi_R\psi\chi\|_{L^2_x}, 
\nonumber
\end{align}
where $a$ is a positive number such that
\begin{equation}\label{eq:conditiona}
a>\frac{\gamma-\delta/2}{1-\delta/2}.
\end{equation}
Here
\begin{equation*}
  \widetilde{M}(y,z) =  \left(|y|/\sqrt{z^{\delta}}\right)^{a} (\nabla \Lambda)\left(y/\sqrt{z^{\delta}}\right).
\end{equation*} 
Observe that since $\gamma\in(1/2,1)$ we have 
\begin{equation}
\frac{\gamma-\delta/2}{1-\delta/2}<1,
\end{equation}
hence we can take take $a=1$.
Explicitly, we have 
\begin{align*}\label{eq:epo62}
\|\widetilde{M}\psi_{R}\psi\chi\|_{L^{2}_{x}}^{2}
 & = \int_{R-R^{\gamma}}^{R+R^{\gamma}}dz\int_{|y| < z}\left| \left(\frac{|y|}{\sqrt{z^{\delta}}}\right)^{a} (\nabla \Lambda)\left(\frac{y}{\sqrt{z^{\delta}}}\right)\right|^{2}dy
 \\
 & = \int_{R-R^{\gamma}}^{R+R^{\gamma}}dz\int_{|y| < \sqrt{z^{2-\delta}}}\left||y|^{a}(\nabla \Lambda)(y)\right|^{q'}z^{\delta}dy
 \\
 &\leq CR^{(\delta+\gamma)/2},
\end{align*}
for all $R>1$ and $1/2<\gamma<1$. Hence by \eqref{eq:soboleveigen} it follows that
\begin{equation}\label{eq:gradlambda}
\|(\nabla\Lambda)\psi_R\psi\chi\|_{L^2_x}\leq CR^{(\delta+\gamma)/2-(a(1-\delta/2)+\delta/2)}
\end{equation}
Finally, if we take $L(y)=v(y)$, from \eqref{eq:norma12}, \eqref{eq:gradpsiR}, \eqref{eq:gradpsichi}, \eqref{eq:gradlambda}  and the condition \eqref{eq:conditiona} it holds
\begin{equation}\label{eq:epo31}
\|\Lambda\psi_{R}\psi\chi\|_{\dot{H}^1_{x}}\leq C R^{(\delta+\gamma)/2-\gamma}.
\end{equation}
In conclusion, \eqref{eq:epo21} follows by interpolation between \eqref{eq:epo30} and 
\eqref{eq:epo31}.

{\bf Proof of \eqref{eq:epo23}.}
In order to conclude the proof of the Lemma, it remains to prove \eqref{eq:epo23}.
Notice that by \eqref{eq:epo29} it turns out that
\begin{equation}\label{eq:epo32}
 \|\Lambda\psi_{R}\psi\chi\|_{L^{q}_{x}}\leq C R^{(\delta+\gamma)/q},
\end{equation}
and with the same argument as above it can be shown that
\begin{equation}\label{eq:epo33}
 \|\Lambda\psi_{R}\psi\chi\|_{\dot{H}^{2}_q}\leq C R^{(\delta+\gamma)/q-2\gamma}.
\end{equation}
Therefore, by interpolation between \eqref{eq:epo32} and \eqref{eq:epo33}, we obtain
\begin{equation}\label{eq:epo34}
 \|\Lambda\psi_{R}\psi\chi\|_{\dot{H}^{2\sigma}_q}\leq C R^{(\delta+\gamma)/q-2\sigma\gamma}.
\end{equation}
Due to the explicit form of $F_R$,
in order to prove \eqref{eq:epo23} we need some control on the two terms in \eqref{eq:FR}. We claim that the following estimates hold
\begin{equation}\label{eq:epo35}
 \|\psi_R\psi\chi F\|_{L^{p}_{T}\dot{H}^{2\sigma}_q}\leq C T^{1/p} R^{(\delta+\gamma)/q-2\sigma\gamma}\max\{R^{-1}, TR^{-(1+\delta/2)}\},
\end{equation}
\begin{equation}\label{eq:epo36}
\|G_R\|_{L^{p}_{T}\dot{H}^{2\sigma}_q}\leq C T^{1/p} R^{(\delta+\gamma)/q-2\sigma\gamma}\max\{R^{-(2-\delta/2)}, R^{-\gamma}, R^{-(3-\delta)}\},
\end{equation}
under the conditions $1/2<\gamma<1$ and $1<\delta<2$. 

\underline{Proof of \eqref{eq:epo35}.}
By \eqref{eq:F}, it is sufficient to estimate terms of the type
\begin{equation}
\frac{t}{z^{1+\delta/2}} \Lambda\psi_{R}\psi\chi,
\qquad
\frac{1}{z} \Lambda\psi_{R}\psi\chi,
\end{equation}
where $\Lambda$ is defined by \eqref{eq:epo28} and L has to be chosen in a suitable way. Let us consider the first term. Due to the properties of $\psi_R$, in the support of $\frac{t}{z^{1+\delta/2}}\Lambda\psi_R\psi\chi$ we have that $t/z^{1+\delta/2}\leq Ct/R^{1+\delta/2}$; hence by \eqref{eq:epo34} we obtain
\begin{align}\label{eq:epo37}
\|\frac{t}{z^{1+\delta/2}} \Lambda\psi_{R}\psi\chi\|_{L^{p}_{T}\dot{H}^{2\sigma}_q} &\leq CT^{1+1/p}\frac{1}{R^{1+\delta/2}}\|\Lambda\psi_{R}\psi\chi\|_{\dot{H}^{2\sigma}_q}
\\
&
 \leq  CT^{1+1/p}R^{(\delta+\gamma)/q-2\sigma\gamma-(1+\delta/2)}.
 \nonumber
\end{align}
By the same arguments, we can deduce that
\begin{equation}\label{eq:eponueva}
\|\frac{1}{z} \Lambda\psi_{R}\psi\chi\|_{L^{p}_{T}\dot{H}^{2\sigma}_q} \leq CT^{1/p}R^{(\delta+\gamma)/q-2\sigma\gamma-1}.
\end{equation}
Finally, \eqref{eq:epo37} and \eqref{eq:eponueva} imply \eqref{eq:epo35}.

\underline{Proof of \eqref{eq:epo36}.}
Denote by
\begin{equation}\label{eq:tetapsi}
\Theta(y,z) := \left(\frac{|y|}{\sqrt{z^{\delta}}}\right)^{2} L\left(\frac{y}{\sqrt{z^{\delta}}}\right),
\end{equation}
for $(y,z) \in \mathbb{R}^{2} \times \mathbb{R}$.
Arguing as in the previous cases, we easily obtain
\begin{align*}
&\|\frac{|y|}{z^2}\Lambda\psi_R\psi'\chi\|_{L^{p}_T\dot{H}^{2\sigma}_q}=\|\frac{|y|}{z^{2-\delta/2}}\Psi\psi_R\psi'\chi\|_{L^{p}_T\dot{H}^{2\sigma}_q}\leq CT^{1/p}R^{(\delta+\gamma)/q-2\sigma\gamma-(2-\delta/2)},
\\
&\|\frac{|y|}{z^2}\Lambda\psi_R\psi\chi'\|_{L^{p}_T\dot{H}^{2\sigma}_q}=\|\frac{|y|}{z^{2-\delta/2}}\Psi\psi_R\psi\chi'\|_{L^{p}_T\dot{H}^{2\sigma}_q}\leq CT^{1/p}R^{(\delta+\gamma)/q-2\sigma\gamma-(2-\delta/2)},
\\
&\|\Lambda\psi^{'}_R\psi\chi\|_{L^{p}_T\dot{H}^{2\sigma}_q}\leq C T^{1/p}R^{(\delta+\gamma)/q-2\sigma\gamma-\gamma},
\\
&\|\frac{|y|^2}{z^3}\Lambda\psi_R\psi^{'}\chi\|_{L^{p}_T\dot{H}^{2\sigma}_q}=\|\frac{1}{z^{3-\delta}}\Theta\psi_R\psi^{'}\chi\|_{L^{p}_T\dot{H}^{2\sigma}_q}\leq CT^{1/p}R^{(\delta+\gamma)/q-2\sigma\gamma-(3-\delta)},
\\
&\|\frac{|y|^2}{z^3}\Lambda\psi_R\psi\chi'\|_{L^{p}_T\dot{H}^{2\sigma}_q}=\|\frac{1}{z^{3-\delta}}\Theta\psi_R\psi\chi'\|_{L^{p}_T\dot{H}^{2\sigma}_q}\leq CT^{1/p}R^{(\delta+\gamma)/q-2\sigma\gamma-(3-\delta)}.
\end{align*}
Therefore \eqref{eq:epo36} follows by \eqref{eq:GR} and the above estimates.

\medskip

Now \eqref{eq:epo23} follows by \eqref{eq:FR}, \eqref{eq:epo35}, and \eqref{eq:epo36}.
Finally, estimates \eqref{eq:epo25} and \eqref{eq:epo26} are immediate consequences of \eqref{eq:epo21}, \eqref{eq:epo22} and \eqref{eq:epo23}, with the choice $T:=R^\beta$.
 \end{proof}

\section{Proof of Theorem \ref{thm:counterdirac}}\label{sec:proof}

We are now ready to prove Theorem \ref{thm:counterdirac}. For any $x\in\R^3$, denote by $x=(y,z)$, where $y=(y_1,y_2)\in\R^2$, $z\in\R$. By homogeneity we have 
\begin{equation}\label{eq:omog}
A(y,z)=z^{1-\delta}A\left(\frac{y}{z}, 1\right),
\end{equation}
for all $(y,z)\in\R^2\times(0,\infty)$. Let us recall the explicit form
\begin{equation}\label{eq:expansion}
\mathcal D_A=-i\alpha\cdot\nabla-\alpha\cdot A.
\end{equation}
Let $P=(0,0,1)$; since $A(P)=0$ and the differential $DA(P)=M$, the first-order Taylor expansion of $A$ around $P$ in \eqref{eq:omog} gives
\begin{align}\label{eq:taylor}
\alpha\cdot A(y,z)
&=z^{1-\delta}\alpha\cdot\left\{M\left(\frac{y}{z}, 0\right)^{t}
+
R_{1}\left(\frac{y}{z}\right)\right\}
\\
&
= z^{-\delta}\alpha\cdot M(y,0)^{t}+z^{1-\delta}\alpha\cdot R_1\left(\frac{y}{z}\right),
\nonumber
\end{align}
where the rest $R_1$ satisfies
\begin{equation}\label{eq:epo42}
\left|R_{1}\left(\frac{y}{z}\right)\right| \leq C\frac{|y|^{2}}{z^{2}},
\end{equation}
for all $(y,z)\in\R^{2}\times(0,+\infty)$ such that $|y|<|z|$.
We can now select a couple $(\lambda, v(y))$ which satisfies the eigenvalue problem \eqref{eq:eigenproblem}; hence from now on the functions $W_R$, $f_R$ and $F_R$ are fixed by \eqref{eq:WR}, \eqref{eq:fR} and \eqref{eq:FR}.

Denote by $u_R$ the solution of \eqref{eq:Diracmagnintr}; 
due to \eqref{eq:expansion} and \eqref{eq:taylor} we can rewrite the initial value problem as follows
\begin{equation}\label{eq:epo69}
i\partial_{t}u_{R}-\left(i\alpha\cdot\nabla+z^{-\delta}\alpha\cdot M(y,0)^t\right)u_R+z^{1-\delta}\alpha\cdot R_{1}\left(\frac{y}{z}\right)u_{R} =\tilde{F}_{R},
\end{equation}
where
\begin{equation}\label{eq:epo69datum}
  u_{R}(0, y, z) = f_{R},
\end{equation}
and
\begin{equation}\label{eq:epo70}
 \tilde{F}_{R}(t,y,z) = \chi_{(0, R^{\beta})}(t)\left\{F_{R}-z^{1-\delta}\alpha\cdot R_{1}\left(\frac{y}{z}\right)W_{R}\right\},
\end{equation}
with $\beta$ the same as in Lemma \ref{lem:lemapre}, and $f_R$, $F_R$ are given by \eqref{eq:fR}, \eqref{eq:FR}. Notice that due to \eqref{eq:sistemaWR} and \eqref{eq:epo70}, $u_R$  coincides with the solution $W_R$ of \eqref{eq:sistemaWR} for small times $t\in(0,R^{\beta})$. We prove the following Lemma. 
\begin{lemma}\label{lem:mainodd}
  Let $(p,q)$ be an admissible couple in the sense of \eqref{eq:admis}, with $(p,q)\neq(\infty,2)$, and let $(p',q')$ be the dual couple. The following estimate holds:
  \begin{equation}\label{eq:epo71}
    \frac{\|W_{R}\|_{L^{p}((0,R^{\beta});L^{q}_{x})}}{\|\tilde{F}_{R}\|_{L^{p'}((0,R^{\beta});\dot{H}^{2\sigma}_{q'})}} \geq CR^{\mu}, 
  \end{equation}
  for all $1/2 < \gamma < 1$,
  where
  \begin{equation}\label{eq:epo72}
     \mu  =\mu(\delta, \gamma, \beta, p) = 2\left(\frac{\beta-(\delta-\gamma)}{p}\right)
     + \min\left\{\gamma - \beta, 1+\frac{\delta}{2}-2\beta, 2-\frac{\delta}{2}-\beta\right\},
  \end{equation}
  and $\gamma$ is the same as in \eqref{eq:psiR}.
\end{lemma}
\begin{proof}
Thanks to \eqref{eq:epo26}, we just need to estimate the rest term in \eqref{eq:epo70}. We easily obtain
\begin{align*}
&\|z^{1-\delta}R_1\left(\frac{y}{z}\right)W_R\|_{L^{p'}((0,T);\dot{H}^{2\sigma}_{q'})}\\
&\ \ \ 
\leq CT^{1/p'}\|\frac{|y|^2}{z^{1+\delta}}\omega\psi_R\psi\chi\|_{\dot{H}^{2\sigma}_{q'}}\leq CT^{1/p'}R^{-1}\|\Theta\psi_R\psi\chi\|_{\dot{H}^{2\sigma}_{q'}}\\
&\ \ \
\leq CT^{1/p'}R^{(\delta+\gamma)/q'-2\sigma\gamma-1},
\end{align*}
for any $R>1$, and $1/2<\gamma<1$, where $\omega$ is the rescaled eigenfunction in \eqref{eq:omega1}. Hence for $t\in(0,R^{\beta})$ we have
\begin{equation}\label{eq:error}
\|z^{1-\delta}R_1\left(\frac{y}{z}\right)W_R\|_{L^{p'}((0,R^{\beta});\dot{H}^{2\sigma}_{q'})}\leq CR^{\beta/p'+(\delta+\gamma)/q'-2\sigma\gamma-1}
\end{equation}
for any $\beta>0$.

By \eqref{eq:epo70}, we have 
\begin{equation}\label{eq:quasi}
  \frac{\|W_{R}\|_{L^{p}((0,R^{\beta});L^{q}_{x})}}{\|\tilde{F}_{R}\|_{L^{p'}
  ((0,R^{\beta});\dot{H}^{2\sigma}_{q'})}}
  \geq
  \frac{\|W_{R}\|_{L^{p}((0,R^{\beta});L^{q}_{x})}}{\|F_{R}\|_{L^{p'}
  ((0,R^{\beta});\dot{H}^{2\sigma}_{q'})}+\|z^{1-\delta}R_1\left(\frac{y}{z}\right)W_R\|_{L^{p'}((0,R^{\beta});\dot{H}^{2\sigma}_{q'})}}.
\end{equation}
In conclusion, \eqref{eq:epo71} follows from \eqref{eq:epo22}, \eqref{eq:epo23}, \eqref{eq:error} and \eqref{eq:quasi} whenever the couple $(p,q)$ satisfies the admissibility condition \eqref{eq:admis}.
\end{proof}

Let us now go back to the inhomogeneous Cauchy problem \eqref{eq:epo69}-\eqref{eq:epo69datum}. Notice that
\begin{equation}\label{eq:epo74}
 \|u_{R}\|_{L^{p}_{t}L^{q}_{x}} \geq \|u_{R}\|_{L^{p}((0, R^{\beta}); L^{q}_{x})} =  \|W_{R}\|_{L^{p}((0, R^{\beta}); L^{q}_{x})},
\end{equation}
from which it follows
\begin{equation}\label{eq:epo75}
 \frac{\|u_{R}\|_{L^{p}_{t}L^{q}_{x}}}{\|f_{R}\|_{\dot{H}^{\sigma}_x} + \|\tilde{F}_{R}\|_{L^{p'}_{t}\dot{H}^{2\sigma}_{q'}}} \geq \frac{ \|W_{R}\|_{L^{p}((0, R^{\beta}); L^{q}_{x})}}{\|f_{R}\|_{\dot{H}^{\sigma}_x} + \|\tilde{F}_{R}\|_{L^{p'}((0, R^{\beta}); \dot{H}^{2\sigma}_{q'})}}
\end{equation} 
Observe that \eqref{eq:epo25} implies that for any $1\leq p<\infty$, $1/2<\gamma<1$, 
\begin{equation}\label{eq:epo76}
\frac{\|W_{R}\|_{L^{p}((0, R^{\beta}); L^{q}_{x})}}{\|f_{R}\|_{\dot{H}^{\sigma}_x}} \to +\infty,
\end{equation}
as $R\to+\infty$, provided that $\beta>(\delta-\gamma)$. On the other hand, the function $\mu(\delta,\gamma,\beta,p)$ defined in \eqref{eq:epo72}, is continuous with respect to $\delta$, and in particular
\begin{equation}
\mu(\delta,\gamma,\delta-\gamma,p)=\min\{2\gamma-\delta, 1-\frac{3\delta}{2}+2\gamma, 2-\frac{3\delta}{2}+\gamma\}.
\end{equation}
Hence $\mu$ is strictly positive if 
\begin{equation}\label{eq:range}
\frac{\delta}{2}<\gamma<1.
\end{equation}
Since $1<\delta<2$, the range given by \eqref{eq:range} contains some $\gamma\in(1/2,1)$; hence, by choosing $\beta>\delta-\gamma$, we obtain $\mu>0$. This remark, together with \eqref{eq:epo71}, gives
\begin{equation}\label{eq:epo78}
\frac{\|W_{R}\|_{L^{p}((0,R^{\beta});L^{q}_{x})}}{\|\tilde{F}_{R}\|_{L^{p'}((0,R^{\beta});\dot{H}^{2\sigma}_{q'})}} \to +\infty,
\end{equation}
as $R\to+\infty$. Therefore, by \eqref{eq:epo75}, \eqref{eq:epo76} and \eqref{eq:epo78} we conclude that 
\begin{equation}\label{eq:epo79}
\frac{\|u_{R}\|_{L^{p}_{t}L^{q}_{x}}}{\|f_{R}\|_{\dot{H}^{\sigma}_x} + \|\tilde{F}_{R}\|_{L^{p'}_{t}\dot{H}^{2\sigma}_{q'}}} \to +\infty,
\end{equation}
as $R\to+\infty$. The last inequality shows that the following Strichartz estimates 
\begin{equation}\label{eq:epo80}
\|u\|_{L^{p}_{t}L^{q}_{x}} \leqslant C\left(\|f\|_{\dot{H}^{\sigma}_x} +  \|F\|_{L^{p'}_{t}\dot{H}^{2\sigma}_{q'}}\right)
\end{equation}
cannot be satisfied by solutions of the inhomogeneous Dirac equation
\begin{equation}\label{eq:epo81}
\begin{cases}
i\partial_{t}u-i\alpha\cdot\nabla_A u=F
\\
u(x,0)=f(x),
\end{cases}
\end{equation}
where the potential $A$ is given by \eqref{eq:A}. In order to disprove Strichartz estimates for the corresponding homogeneous Dirac equation, near the point $p=\infty$, $q=2$, it is sufficient to apply a standard $TT^{*}$-argument, via Christ-Kiselev Lemma (see \cite{CK}). The rest of the estimates fail by interpolation with the mass conservation (i.e. the $L^\infty L^2$-estimate). This completes the proof of Theorem \ref{thm:counterdirac}.


\begin{thebibliography}{abc12e}

\bibitem{BDF}
{\sc Boussaid, N., D'Ancona, P., and Fanelli, F.}, Virial identity and weak dispersion for the magnetic dirac equation. \textit{Journ. Math. Pures Appl.} \textbf{95} (2011), 137--150.

\bibitem{BPST}
{\sc Burq, N., Planchon, F., Stalker, J., and Tahvildar-Zadeh, S.},
Strichartz estimates for the wave and Schr\"odinger equations with
the inverse-square potential, \textit{J. Funct. Anal.} \textbf{203}
(2003) no. 2, 519--549.

\bibitem{BPST2}
{\sc Burq, N., Planchon, F., Stalker, J., and Tahvildar-Zadeh, S.}
Strichartz estimates for the wave and {S}chr\"odinger equations with
potentials of critical decay, \textit{Indiana Univ. Math. J.}
\textbf{53}(6) (2004), 1665--1680.

\bibitem{CK}
{\sc Chirst, M., and Kiselev, A.}, Maximal functions associated to
filtrations, \textit{J. Funct. Anal.} \textbf{179}(2) (2001),
409--425.

\bibitem{DF2}
{\sc D'Ancona, P., and Fanelli, L.},
Decay estimates for the wave and Dirac equations with a magnetic potential,
\textit{Comm. Pure Appl. Math.} {\bf 60} (2007), 357--392.

\bibitem{DF}
{\sc D'Ancona, P., and Fanelli, L.},
Strichartz and smoothing estimates for dispersive equations with magnetic
potentials, \textit{Comm. Part. Diff. Eqns.} {\bf 33} (2008), 1082--1112.

\bibitem{DFVV}
{\sc D'Ancona, P., Fanelli, L., Vega, L., and Visciglia, N.}, Endpoint Strichartz estimates for the magnetic
Schr\"odinger equation, \textit{J. Funct. Anal.} {\bf 258} (2010), 3227--3240.

\bibitem{D}
{\sc Dirac, P.}, The Quantum Theory of the Electron,
\textit{Proc. R. Soc. Lond. A} {\bf 117} (1928), 610--624.

\bibitem{Du}
{\sc Duyckaerts, T.}, A singular critical potential for the Schr\"odinger operator, \textit{Canad. Math. Bull.} {\bf 50}(1) (2007), 35--47.

\bibitem{EGS1}
{\sc Erdogan, M.B., Goldberg, M., and Schlag, W.},
Strichartz and Smoothing Estimates for Schr\"odinger
Operators with Almost Critical Magnetic Potentials in Three and Higher
Dimensions,
\textit{Forum Math.} {\bf 21} (2009), 687--722.

\bibitem{EGS}
{\sc Erdogan, M.B., Goldberg, M., and Schlag, W.},
Strichartz and smoothing estimates for Schrodinger operators with
large magnetic potentials in $\R^3$, \textit{J. European Math.
Soc.} {\bf 10} (2008), 507--531.

\bibitem{FG}
{\sc Fanelli, L., and Garc\'ia, A.}, Counterexamples to Strichartz estimates for the magnetic Schr\"odinger equation, \textit{Comm. Cont. Math.} {\bf13}(2) (2011), 213--234.

\bibitem{FV}
    {\sc Fanelli, L., and Vega, L.},
    Magnetic virial identities, weak dispersion and Strichartz
    inequalities, \textit{Math. Ann.} {\bf 344} (2009), 249--278.
    
\bibitem{GST}
{\sc Georgiev, V., Stefanov, A., and Tarulli, M.} Smoothing -
Strichartz estimates for the Schr\"odinger equation with small
magnetic potential, \textit{Discrete Contin. Dyn. Syst.} A {\bf 17}
(2007), 771--786.   

\bibitem{GV}
{\sc Ginibre, J., and Velo, G.}, Generalized Strichartz inequalities
for the wave equation, \textit{J. Funct. Anal.} \textbf{133} (1995)
no. 1, 50--68.

\bibitem{G}
{\sc Goldberg, M.}, Strichartz estimates for Schr\"odinger operators 
with a non-smooth magnetic potential, {\sc arXiv:0804.0034v1} (2008). 

\bibitem{GVV}
{\sc Goldberg, M., Vega, L., and Visciglia, N.}, Counterexamples of
Strichartz inequalities for Schr\"odinger equations with repulsive
potentials, \textit{Int. Math Res Not.}, 2006 Vol. 2006: article ID
13927.

\bibitem{KT}
{\sc Keel, M., and Tao, T.}, Endpoint Strichartz estimates,
\textit{Amer. J. Math.} \textbf{120} (1998) no. 5, 955--980.

\bibitem{KM}
{\sc Klainerman, S., and Machedon, M.}, Space time estimates for null forms and the local
existence theorem, \textit{Comm. Pure Appl. Math.} {\bf 46} (1993), 1221--1268.

\bibitem{M}
{\sc Montgomery-Smith, S.J.}, Time Decay for the Bounded Mean Oscillation of Solutions of the
Schr¬odinger and Wave Equation, \textit{Duke Math J.} {\bf 19} (1998), 393--408.

\bibitem{RV}
{\sc Ruiz, A., and Vega, L.}, On local regularity of Schr\"odinger
equations. \textit{Int. Math. Research Notices} 1, 1993, 13--27 .

\bibitem{ST}
{\sc Staffilani, G., and Tataru, D.}, Strichartz estimates for a Schr\"odinger operator with nonsmooth
coefficients, \textit{Comm. Part. Diff. Eq.}, {\bf 27} (2002), 1337--1372.

\bibitem{Ste}
{\sc Stefanov, A.}, Strichartz estimates for the magnetic
Schr\"odinger equation, \textit{Adv. Math.} {\bf 210} (2007), 246--303.

\bibitem{S}
{\sc Strichartz, R.}, Restriction of Fourier transforms to quadratic surfaces and decay of solutions of wave equations, \textit{Duke Math. J.} {\bf 44} (1977), 705--774.


\bibitem{T}
{\sc Thaller, B.}, \textit{The Dirac Equation}, Springer-Verlag 1992.




\end{thebibliography}
\end{document}